\documentclass{svproc}
\usepackage{url}

\usepackage{graphicx}
\usepackage[latin1]{inputenc}
\usepackage[bookmarksnumbered,colorlinks]{hyperref}
\usepackage{pgf,pgfarrows,pgfnodes,pgfautomata,pgfheaps,pgfshade}
\usepackage{pgfplots}
\usepackage{dsfont}
\usepackage{tikz-3dplot}
\usepackage{tikz}
\usepackage{enumerate}
\usepackage{amsfonts,amssymb}
\usepackage{amsmath}
\usepackage[all]{xy}
\usepackage[normalem]{ulem} 

\usepackage[export]{adjustbox}

\newcommand{\C}{\mathcal{C}}
\newcommand{\R}{\mathds{R}}  
\newcommand{\N}{\mathds{N}}

\newcommand{\de}{\mathrm d}
	
\newcommand{\cstk}{{\rm CSTK} }

\begin{document}
\mainmatter              
\title{Causal ladder of Finsler spacetimes with a \\ cone Killing vector field}
\titlerunning{Causal ladder of Finsler spacetimes with a cone Killing vector field}  
%
\author{Erasmo Caponio\inst{1} \and Miguel  Angel Javaloyes\inst{2}} 
\authorrunning{M. A. Javaloyes et al.} 
%
\tocauthor{Erasmo Caponio, Miguel \' Angel Javaloyes and Miguel S\' anchez}
\institute{ Politecnico di Bari, Dipartimento di Meccanica, Matematica  e Management,  Italy\\
	\email{erasmo.caponio@poliba.it}
\and	
	University of Murcia, Department of Mathematics, Spain\\
\email{majava@um.es}
}

\maketitle              
\begin{center}
	\it \footnotesize To the memory of Franco Mercuri
\end{center}

\vspace{0.2cm}
	\begin{abstract} 
 The correspondence between wind Riemannian structures  and spacetimes endowed with a Killing vector field as developed in full generality in \cite{CJS14} is deepened  by considering  a cone structure endowed with a vector field that preserve the structure  
(termed {\em cone Killing vector field}) and a wind Finslerian structure, introduced in \cite{CJS14} as well.
Causality properties of the former are characterized by  using metric-type properties of the latter.
A particular attention is posed to the case of a cone  structure associated with a  Finsler-Kropina type metric, i.e. a field of compact and  strongly convex indicatrices that 
enclose the zero vector in the closure of its bounded interior
at each tangent space of the manifold.	
\end{abstract}
\section{Introduction}

The theory of causality has been present in General Relativity since its first days. Indeed, non-chronological spacetimes have been seen as a pathology since the beginning. But it was not until the 60s that causality reached its peak  hand in hand with the singularity theorems of Penrose and Hawkings.  Several tools have been developed to study the causal properties of a spacetime as for example Penrose diagrams. But, in some particular cases, namely, when the metric admits a Killing vector with a spacelike section, causality can be encoded in an associated Finsler metric on the section. The first results on this line established a characterization of the causal ladder of a standard stationary spacetime $(\R\times S,g)$ in terms of completeness properties and geodesic convexity of a certain Randers metric on   $S$ (see \cite{CJM11,CJS10}). Let us recall that a Randers metric is characterized by having a non-necessarily centered ellipsoid as indicatrix which contain the origin in the interior of the bounded domain enclosed at every tangent space. 
Later on this has been generalized to stationary Finsler spacetimes (see \cite{CapSta16,CapSta18}) allowing for arbitrary Finsler metrics on the section, and in this way, removing the restriction of considering Randers metrics. 

Let us return to the case of classical Lorentzian spacetimes.
When the Killing vector field is not timelike, the Randers metric must be replaced with a wind Riemannian structure, which is determined by a field of ellipsoids, but this time with origin that can be placed everywhere.
The name of wind Riemannian structures was coined regarding the Zermelo problem of navigation.  As a matter of fact, the field of indicatrices can represent the prescribed velocities of a moving object, a boat or a zeppelin, which in principle has the same speed in all directions, but in the presence of a wind, these velocities are given by the translation of a sphere. If the wind is not strong these translated spheres determine a Randers metric. But when it is strong enough, namely, the push of the wind is stronger than the engine of the boat, then we obtain a wind Riemannian structure.
In particular, when the Killing vector field is spacelike, the wind Riemannian structure has the origin outside the bounded region determined at every tangent space by the field of ellipsoids and, as a consequence, there are two pseudo-Finsler metrics associated with the wind Riemannian structure. Anyway, it is still  possible to characterize the causality of the spacetime using both metrics as it was done in \cite{CJS14}. 

The main result of this work is to generalize the results concerning the causal ladder in \cite{CJS14} to Finsler spacetimes. But as causality depends only on the future lightcone of the Finsler spacetime, we will work in the more general setting of cone structures (see Definition \ref{d_conestructure}). Observe that every cone structure is the future lightcone of a Finsler spacetime, but this Finsler spacetime is not unique. Rather than a Killing vector field of the Finsler spacetime, we will consider a vector field with flow that preserves the cone structure. This will be called a {\em cone Killing field} of the cone structure, and when it admits a spacelike section, it will be possible to associate a wind Finslerian structure  (a field of compact, strongly convex hypersurfaces) that encodes the causality of the cone structure. Again, when the origin is outside the bounded region determined by the field of strongly convex hypersurfaces, we have two associated pseudo-Finsler metrics, which in turn allow us to characterize all the causal properties in terms of its completeness and convex properties (see Theorem \ref{generalK}). In a previous section, we have considered the distinguished case in that the cone Killing field $K$ is causal and there is only one associated (conic) Finsler metric on the section. This metric  is of Kropina type where $K$ is lightlike and a classical Finsler metric where $K$ is timelike.  The causal ladder  can then be characterized in terms of this Finsler-Kropina type metric (see Theorem \ref{kropinaLadder}).

We note that the definition of a cone structure in Definition~\ref{d_conestructure} is equivalent to the one in \cite{JS18} (see Remark~\ref{equivalence}). Specifically, we define a cone structure as an embedded hypersurface $\C$ in $TM$ representing the set of  ``future lightlike vectors'' at each point $p\in M$ that satisfies a transversality condition and some convex properties. A field of cones with convex, non-empty interiors is naturally associated with $\C$, defined as $p\mapsto A_p\cup \C_p$.
In other works, starting from the seminal paper \cite{FS12} and continuing in \cite{Min19}, a cone structure is defined as a set-valued map $p\in M\mapsto \C_p\subset T_p M$, where $\C_p\cup{0}$ is closed in $T_p M$ (not necessarily with non-empty interior in \cite{Min19}) that satisfies certain regularity conditions, which vary slightly across the references. A more general definition is provided in \cite{BS18}, where a cone structure $\C$ is introduced as a topological subset of $TM$ such that $\C\cup \bf 0$ is a closed subset, and $\C\cap T_pM$, for each $p\in M$, is a (possibly empty or with empty interior) convex cone that can also be equal to the whole space $T_p M$.
While the latter two works present a more general notion of a cone structure,
Definition~\ref{d_conestructure} is tailored to the concept of  wind Finslerian structure.

The paper is structured as follows. In \S \ref{finsterm} we collect some basic definitions concerning Finsler geometry. In \S \ref{conestruct} we introduce the concept of cone structure and its causality \S \ref{causalityforcones}. Then we associate a cone triple to a cone structure \S \ref{conetriple}, which will be used to define a Finsler spacetime having the cone structure as its future lightcone \S \ref{Finslight}. In the next two subsections, we explore the extension of the notion of limit curve to cone structures and we make a summary of the results about causality that we will use and how they can be extended to cone structures in a simple way. In 
\S \ref{wfm}, we introduce the notion of wind Finslerian structure  and the main concepts associated with it as wind balls, c-balls... In \S \ref{windtriples} we adapt the notion of cone triples to the case in that we consider an arbitrary vector field in the triple (not necessarily timelike).  In \S \ref{cstk} we introduce the notion of cone Killing fields and cone structures with a space-transverse Killing field splitting (\cstk splitting  in short). Finally, in \S \ref{FinslerKropina} and \S \ref{arbitraryKilling}, we study the causal ladder of a \cstk splitting, first when the cone Killing field is causal and then when it is arbitrary.

\section{Finslerian terminology}\label{finsterm}
Let us introduce some terminology concerning Finsler geometry in a more general setting than the standard one considered for example in \cite{BCS01,ChSh05}. A {\em (two-homogeneous) pseudo-Minkowski norm} in a vector space $V$ is defined as a function $L:A\subset V\rightarrow \R$ such that

\begin{enumerate}[(i)]
	\item $A$ is a conic open subset, namely, if $v\in A$, then $\lambda v\in A$ for all $\lambda>0$,
	\item it is smooth away from the zero vector,
	\item it is positive homogeneous of degree two, namely, $L(\lambda v)=\lambda^2 L(v)$ for all $v\in A$ and $\lambda>0$,
	\item the Hessian
	\begin{equation}\label{fundten}
	g_v(u,w)=\frac{1}{2} \frac{\partial^2}{\partial t\partial s} L(v+tu+sw)|_{(t,s)=(0,0)},
	\end{equation} 
	for all $v\in A\setminus \{0\}$ and $u,w\in V$, is non-degenerate.
	
\end{enumerate}
Moreover, given a pseudo-Minkowski norm with $L>0$,  we can consider $F=\sqrt{L}$ (homogeneous of degree one) and we say that $F$ is
\begin{enumerate}[(i)]
	\item {\em a conic Minkowski norm } if $g_v$ in \eqref{fundten} is positive definite for all $v\in A$,
	\item {\em a Lorentz-Minkowski norm} if $g_v$ in \eqref{fundten} has index $n-1$ for all $v\in A$, where $n=\dim V$,
	\item {\em a Minkowski norm} if $A=V$ and \eqref{fundten} is positive definite for all $v\in A$.	
	\end{enumerate}
In all the cases, the subset 
\begin{equation}\label{indicatrix}
\Sigma=\{v\in A: F(v)=1\}
\end{equation}
is called the 
{\em indicatrix} of $F$. 

\begin{remark}\label{sff}
	Observe that the fundamental tensor in \eqref{fundten} restricted to the indicatrix $\Sigma$ coincides with the second fundamental form of $\Sigma$ at $v$ with respect to the vector $\xi=-v$ and the Levi-Civita connection associated with any 
	positive inner product\footnote{There are many choices of positive definite inner products in $V$, so many as ellipsoids centered at the origin, but  all of them determine the same Levi-Civita connection.}  on $V$ (see \cite[Eq. (2.5)]{JS14} and observe that $\xi(G)=-1$, where $G=\frac 1 2 F^2$).

	Taking into account that $v$ is $g_v$-orthogonal to $\Sigma$ and $g_v(v,v)=F(v)^2$, it follows that $g_v$ is positive definite if and only if the second fundamental form of $\Sigma$  with respect to $\xi$ is positive definite. As a consequence, any hypersurface $\Sigma$ with this property about the second fundamental form and with the property that any half-ray from the origin intersects $\Sigma$ at most once determines a conic Minkowski norm (see also \cite[Th. 2.14]{JS14}).
	\end{remark}

A particular case of conic Minkowski norm is {\em a Kropina type norm}. It is characterized by an indicatrix  $\Sigma$, whose closure $\bar \Sigma$ is an embedded, compact, strongly convex\footnote{Strongly convex means that its second fundamental form is definite. Moreover, to be definite does not depend on the transversal vector used to compute the second fundamental form.} hypersurface,  diffeomorphic to a sphere and equal to $\bar \Sigma=\Sigma\cup\{0\}$. A typical example of Kropina type norm on $V$ is given by 
\begin{equation}\label{kropina}
	 F(v):=\frac{F^2_0(v)}{\beta(v)},
\end{equation}
where $F_0$ is a Minkowski norm and $\beta$ a co-vector, both on $V$ (see \cite[\S4.2.2]{JS14}). In such a case $A$ is the open half-space defined by $\beta>0$.

The next step is to consider a manifold $M$. Let $A\subset TM$ be an open subset such that for each $p\in M$, $A_p:=A\cap T_pM$ is a nonempty,  conic open subset of $T_pM$. Then a function $L:A\subset TM\rightarrow \R$ is said to be a {\em pseudo-Finsler metric} in $M$ if it is smooth away from the zero section 
and the restriction $L_p:=L|_{A_p}$ is a pseudo-Minkowski norm on $T_pM$ for every $p\in M$. In particular,  given a pseudo-Finsler metric, let us assume that  $L>0$ and $F:=\sqrt{L}$; then  we say that 
\begin{enumerate}[(i)]
	\item $F$ is a {\em conic Finsler metric} if $F_p:=F|_{A_p}$ is a conic Minkowski norm for all $p\in M$,
	\item 
	$F$ is a  {\em Lorentz-Finsler metric} if $F_p:=F|_{A_p}$ is a Lorentz-Minkowski norm for all $p\in M$,
	\item $F$ is a  {\em Finsler metric} if $A_p=T_pM$ and $F_p:=F|_{T_pM}$ is a  Minkowski norm for all $p\in M$.
\end{enumerate}
\section{Cone structures and Finsler spacetimes}\label{conestruct}
The framework of cone structures provides a foundation for defining causality with extensive generality in the context of a smooth manifold. If we remain in the smooth  setting and $\dim V>2$\footnote{The case $\dim V=2$ is trivial in the sense that the cone structure is given by two straight lines that intersect in the origin. }, we can define a {\em strong cone} $\C_0$ in a vector space $V$ as a smooth,  connected, embedded, hypersurface $\C_0$ of $V\setminus \{0\}$ which is 
\begin{enumerate}[(i)]
\item conic, in the sense that if $v\in\C_0$ then $\lambda v\in \C_0$ for all $\lambda>0$, 
\item and such that there exists an affine hyperplane $\Pi$ that does not contain the origin  and  intersects $\C_0$ in a compact,  strongly convex,  embedded smooth hypersurface of $\Pi$ diffeomorphic to a sphere. 
\end{enumerate}
\begin{remark}\label{equivalence}
	It is not difficult to check that this definition coincides with the one given in \cite[Def. 2.1]{JS18}, namely, it is conic, salient (if $v\in \C_0$, then $-v\notin \C_0$), with  convex interior (i.e. it encloses a  convex region) and it is strongly convex in the non-radial directions (see \cite[Lem.2.5 and Prop. 2.6]{JS18}). The last condition can be weakened to convex,  getting the  notion of a weak cone (we will not use it here). 
\end{remark}
\begin{definition}\label{d_conestructure} Let $M$ be a manifold of dimension $n\geq 2$. A {\em (strong) cone structure}
$\mathcal{C}$ is an embedded hypersurface of $TM$ such that, for  each $p\in M$: 

\begin{enumerate}[(a)]	
\item[(a)]  $\mathcal{C}$ is  transverse to  the fibers  of the tangent bundle,  that is, if $v\in \C_p$, then $T_v(T_pM)+T_v\C=T_v(TM)$, and \item[(b)]  each $\mathcal{C}_p:=T_pM\cap \mathcal{C}$ is a 
 strong cone  
  in $T_pM$. 
\end{enumerate}
 The inner (convex) domain enclosed by $\C_p$ will be denoted by $A_p$ and 
$A:=\cup_{p\in M}A_p$, which  will be called the {\em cone domain}. 
\end{definition}
The role of the transversality  condition in the above definition is related to the smoothness of some Finsler metrics associated with the cone structure, see \cite[Th. 2.17]{JS18} (compare  also with \cite[Prop. 2.12]{CJS14}).

\subsection{Causality of a cone structure}\label{causalityforcones}
After this definition,  the causality of a cone structure $(M,\C)$ is analogous to the causality associated with the (future-directed) cone structure of a relativistic (Lorentzian) spacetime. Namely, a vector $v\in TM$ is said to be  {\it future timelike} if $v\in A$, {\it future lightlike} if $v\in \C$, and {\it future causal} if $v\in \C\cup A$.  We say that $v\in TM$ is {\em past} timelike (resp. lightlike, causal) if $-v$ is future timelike (resp. lightlike, causal). Moreover, we say that $v\in TM$ is {\em spacelike} if it is neither future nor past causal.   We say that a piecewise smooth curve $\gamma:[a,b]\rightarrow M$ is future or past timelike (resp. lightlike, causal), if $\dot\gamma(t)$
is future or past timelike (resp. lightlike, causal) for all $t\in [a,b]$. Finally, given two points $p,q\in M$, we say that $q$ is {\em chronologically} (resp. {\em strictly causally}) related to $p$, denoted $p\ll q$ (resp. $p< q$) if there exists a future timelike (resp. causal) curve $\gamma:[a,b]\rightarrow M$ such that $p=\gamma(a)$ and $q=\gamma(b)$. We say that $q$ is {\em causally} related to $p$, denoted with $p\leq q$, if $p< q$ or $p=q$.  The {\em chronological future (resp. past)} of a point $p\in M$ is defined as
\[I^+(p)=\{q\in M: p\ll q\}\quad \text{  (resp. $I^-(p)=\{q\in M: q\ll p\}$)},\]
while the {\em causal future (resp. past)} of a point $p\in M$ is defined as
\[J^+(p)=\{q\in M: p\leq q\}\quad \text{  (resp. $J^-(p)=\{q\in M: q\leq p\}$)}.\]

We can reproduce the main steps of the causal ladder in this context. Let $\mathcal P(M)$ be  the set of parts of $M$,
and, for every compact $K$ of $M$, define the subset $A_K$ of ${\mathcal P}(M)$ as
\[A_K= \{U\in \mathcal P(M): U\cap K=\emptyset\}.\]
 Then it is not difficult to see that the collection of subsets of $\mathcal P(M)$  given by  $\{A_K\subset {\mathcal P(M)}: \text{$K\subset M$, with $K$ compact}\}$ is the base of a topology of $\mathcal P(M)$. Now consider the map
\[I^\pm: M\rightarrow \mathcal P(M).\]
We say that $I^\pm$ is {\em outer continuous} at some $p\in M$ if for any compact $K\subset M\setminus  I^\pm(p)$ there exists an open neighborhood $U$ of $p$ such that $K\subset M\setminus  I^\pm(q)$ for all $q\in U$ (see \cite{HS74}). It is not difficult to check that $I^\pm$ is outer continuous if and only if it is continuous when $\mathcal P(M)$ is endowed with the topology defined above.

\begin{definition}\label{causalprop}
	A cone structure $(M,\C)$ is said to be
	\begin{enumerate}[(i)]
	\item {\em chronological} if there is no point $p\in M$ satisfying that $p\ll p$ (there are no closed timelike curves);
	\item {\em causal } if there is no point $p\in M$ satisfying that $p< p$ (there are no closed causal curves);
	\item {\em future (resp. past) distinguishing} if given $p,q\in M$, $p\not=q$ implies that $I^+(p)\not=I^+(q)$ (resp. $I^-(p)\not=I^-(q)$);
	\item {\em distinguishing} if it is past and future distinguishing, i.\,e., the set valued maps $I^\pm :M\rightarrow {\mathcal P}(M)$ are injective;
	\item {\em strongly causal} if for every $p\in M$ and $U$ an open neighborhood of $p$, there exists a neighborhood $V$ of $p$ included in $U$ such that any causal  curve from $V$ that leaves $U$ does not return to $V$;
	\item {\em stably causal} if there exists a time function, namely, a continuous function $t:M\rightarrow\R$ such that it is strictly increasing along any future  causal curve;
	\item {\em causally continuous} if the maps $I^\pm :M\rightarrow {\mathcal P}(M)$ are injective and continuous with the topology defined above;
	\item {\em causally simple} if it is distinguishing and the subsets $J^\pm(p)$ are closed for all $p\in M$;
	\item {\em globally hyperbolic} if it is strongly causal and $J^+(p)\cap J^-(q)$ is compact for all $p,q\in M$.
	\end{enumerate}
	\end{definition}
\begin{remark} 
	Let us observe that some of the conditions of the causal steps can be replaced with weaker conditions, but we will not go into details. Up to the classical Lorentzian case, the ``distinguishing condition" in the definition of causal simplicity  can be weakened into ``causality" (see \cite[Prop. 3.64]{MinSan}), and the ``strong causality'' of global hyperbolicity can be weakened into ``causality'' (see \cite{BS07}) or even  removed in the non-compact case and in dimension greater than $2$ (see \cite{MH19}). On the other hand, in the reference \cite[Def. 2.23]{Min19} some of the causal properties are defined with slight variations. Even though in that reference the author shows that the usual causal ladder of  spacetimes extends to cone structures, i.e. each of the above properties implies the subsequent ones, see \cite[Th. 2.47]{Min19}, and indeed, it is valid for the more general notion of closed cone structures,  we  will indicate how to prove some basic properties that we need in \S\ref{somebasic} to make the paper more self-contained and to avoid clutter with the different definitions.
\end{remark}
Finally, causal properties allow us to introduce a notion of geodesic which generalizes the notion of lightlike geodesics (see \cite[Def. 2.9]{JS18}).
\begin{definition}\label{conegeos}
Given a cone structure $(M,\C)$, a causal curve $\gamma:[a,b]\rightarrow M$ is said to be a cone geodesic if it is locally horismotic, namely, for every 
$t_0\in[a,b]$ there exists $\varepsilon>0$, and a small neighborhood $U$ of $\gamma(t_0)$ such that for all $s,s'\in I_\varepsilon:=[a,b]\cap (t_0-\varepsilon,t_0+\varepsilon)$, $\gamma(s)\leq_U \gamma(s')$
and $\gamma|_{I_\varepsilon}\subset U$,
 but there is no timelike curve from $\gamma(s)$ to $\gamma(s')$ contained in $U$.
\end{definition}
	\subsection{Cone triples}\label{conetriple}
	Given a cone structure $(M,\C)$,  a non-vanishing one-form $\Omega$ and a vector field $T$ both on $M$ such that $\ker \Omega$ is made of spacelike vectors, $T$ is future timelike everywhere and $\Omega(T)=1$, we have a natural splitting $TM={\rm span}(T)\oplus \ker\Omega$ with a projection $\pi:TM\rightarrow \ker(\Omega)$ determined by
	\[v_p=\Omega(v_p)T_p+\pi(v_p),\quad\text{for all $v_p\in T_pM$ and $p\in M$}.\]
	Every cone triple determines a Finsler metric on the fiber bundle $\ker \Omega$, namely, a function $F:\ker\Omega\rightarrow [0,+\infty)$   such that $F|_{T_pM\cap \ker\Omega}$ is a Minkowski norm for all $p\in M$. 
	This Finsler metric is determined by the following property (see \cite[Th. 2.17]{JS18}):
	\begin{equation}\label{defF}
	v_p\in\C_p \text{ if and only if } v_p=F(\pi(v_p))T_p+\pi(v_p).
	\end{equation}
	Observe that the indicatrix of $F$, $\Sigma_p=\{v\in \ker\Omega\cap T_pM: F(v)=1\}$, when affinely translated into the hyperplane $\{\Omega=1\}\cap T_pM$, coincides with the intersection of that hyperplane with $\C_p$.
	\begin{definition}
		Given a cone structure $(M,\C)$, a non-vanishing one-form $\Omega$ with spacelike kernel and a future  timelike vector $T$ such that $\Omega(T)=1$, we say that $(\Omega,T,F)$ is a cone triple for $(M,\C)$, being $F$ the Finsler metric on $\ker\Omega$ determined by \eqref{defF}.
		\end{definition}
	Observe that an arbitrary cone triple $(\Omega,T,F)$ with $\Omega(T)=1$ uniquely determines a cone structure using \eqref{defF} (see the converse in \cite[Th. 2.17]{JS18}), but there are many cone triples associated with a given cone structure.
	
\subsection{Finsler spacetimes and its lightlike cone structures}\label{Finslight}
A Finsler spacetime is a manifold $M$ together with a Lorentz-Finsler metric  $L:A\rightarrow (0,+\infty)$ such that $A_p$ is also convex for every $p\in M$, and
$L$  can be smoothly extended as $0$ to the boundary of $A$ in $TM\setminus 0$ with
	 vertical Hessian defined as in \eqref{fundten}
that remains of index $\dim M-1$ also at vectors $v$ in the boundary. The set $\{ v\in \bar A\setminus 0:L(v)=0\}$ will be called  the {\em future lightlike cone} of $L$.

Observe that there are several definitions of Finsler spacetimes (see \cite[Appendix A]{JS18}), but with the above definition, the lightlike cone of $L$ is a cone structure (see \cite[Cor. 3.7]{JS18}). Moreover, the converse is true even if there is no uniqueness at all.
\begin{proposition}\label{existsFinsler}
	Given a cone structure $(M,\C)$, there exists a Lorentz-Finsler metric $L:TM\rightarrow (0,+\infty)$ such that $(M,L)$ is a Finsler spacetime with future lightlike cone $\C$.
	\end{proposition}
\begin{proof}
	It can be found in \cite[Cor. 5.8]{JS18}. The sketch of the proof is as follows. First it is possible to find a one-form $\Omega$ such that $\Omega(v)>0$ for all causal vectors of $(M,\C)$, and a future timelike vector field  $T$, such that $\Omega(T)=1$ (see \cite[Lemma 2.15]{JS18}). This completely determines a cone triple $(\Omega, T,F)$ associated with $(M,\C)$. Then one can define the Lorentz-Finsler metric 
	\begin{equation}\label{Gmetric}
	G(tT_p+w_p)=t^2-F(w_p)^2,\quad \forall t\in\R,\forall w_p\in \ker(\Omega_p), \forall p\in M.
	\end{equation}
	Notice that the Lorentz-Finsler metric $G$ is not smooth on the vectors proportional to $T_p$, but it can be smoothened in a small convex conic neighborhood of this direction with closure contained in $A$ (see \cite[Th. 5.6]{JS18}).
	\end{proof}
A Lorentz-Finsler metric determines a nonlinear connection and, then, geodesics  (see \cite[\S4]{AJ16}), an exponential map (smooth away from $0$) and normal neighborhoods (see \cite[\S6.1]{JS18}). 
We call the geodesics $\gamma:I\to M$ defined by a Lorentz-Finsler metric {\em $L$-geodesics}. They satisfy  the conservation of energy $L(\dot\gamma)\equiv\mathrm{const.}$ and, in particular, the ones  satisfying $L(\dot\gamma)\equiv 0$ (resp.  $L(\dot\gamma)\equiv E\in (0, +\infty)$) are called {\em lightlike $L$-geodesics} (resp. {timelike $L$-geodesics})\footnote{By definition $L$ is defined and smooth in a wider cone domain  $A^*\supset \bar A$.  In some cases, as for the smoothened Lorentz-Finsler metric $G$ in \eqref{Gmetric}, the maximal $A^*$ is $TM\setminus 0$. We implicitly assume that  lightlike and timelike $L$-geodesics are only the ones whose velocity vector is in the closure of the privileged  cone domain $A$ associated with $\C$.}. A timelike or lightlike $L$-geodesic is also called a {\em causal $L$-geodesic}.

An open subset which is normal for all its points is called a convex neighborhood (of any of these points).
\begin{proposition}\label{convexneigh}
	Given a cone structure $(M,\C)$ and a point $p\in M$, there exists a neighborhood $U$ of $p$ and a Lorentz-Finsler metric $L:TU\rightarrow \R$ with future lightlike cone given by $\C\cap TU$ such that $U$ is convex for this metric.
	\end{proposition}
\begin{proof}
The only thing one needs is to find a Lorentz-Finsler metric $L:TU\rightarrow \R$ with future lightlike cone given by $\C\cap TU$. 

This can be done using \eqref{Gmetric} and the smoothening procedure in \cite[Th. 5.6]{JS18} applied to $T_p$ and to $-T_p$. 
Then one can use a classical result by Whitehead (see \cite{whitehead}) about the existence of convex neighborhoods.
	\end{proof}
The auxilliary Lorentz-Finsler metric obtained in Proposition \ref{existsFinsler} can be used to compute the cone geodesics (recall Definition~\ref{conegeos}) of the cone structure.
\begin{proposition}\label{coneLgeos}
Let $(M,\C)$ be a cone structure and $L$ a Lorentz-Finsler metric having $\C$ as a future lightlike cone. Then the cone geodesics of $(M,\C)$ are the lightlike pregeodesics of $(M,L)$.
\end{proposition}
\begin{proof}
	See \cite[Th. 6.6]{JS18}.
\end{proof}

\subsection{Limit curves}\label{limitcurves}
As a first step to introduce limit curves in cone structures, we will need to extend the notion of a causal curve to the class of continuous curves.
\begin{definition}
Given a cone structure $(M,\C)$ and a continuous curve $\gamma:[a,b]\rightarrow M$, we say that $\gamma$ is causal (resp. timelike) if for any $t_0\in[a,b]$ there exists a convex neighborhood of $\gamma(t_0)$ as in Proposition \ref{convexneigh} and $\epsilon>0$ such that for all $t,s\in (t_0-\epsilon,t_0+\epsilon)\cap [a,b]$, and $t<s$, the $L$-geodesic connecting $\gamma(t) $ to $\gamma(s)$ is causal (resp. timelike).
\end{definition}
Even if the last definition seems to have a dependence on the metric $L$, the fact that the lightlike pregedeodesics of all the possible $L$ coincide (recall Proposition~\ref{coneLgeos}) implies that it only depends on the cone structure.

Moreover, let us recall the definition of limit curve.
\begin{definition}
Given a manifold $M$, we say that a curve $\gamma:I\rightarrow M$ is a {\em limit curve} of a sequence of curves $\{\gamma_n\}$ if for all $p$ in the image of $\gamma$ and every neighborhood  $W$ of $p$,  there exists a subsequence $\{\gamma_m\}$ of $\{\gamma_n\}$ such  that   all but a finite number of the curves in the subsequence intersect $W$.
\end{definition}
We have already all the ingredients to prove the existence of limit curves for causal curves in a cone structure.
\begin{proposition}\label{existslimitcurve}
Let $(M,\C)$ be a strongly causal cone structure. Then any limit curve of causal curves is also causal.
\end{proposition}
\begin{proof}
Analogous to the Lorentzian one in \cite[Lem. 3.29]{BeEhEa96}, but we will replace the choice of convex neighborhoods which are causally convex, whose existence is not proven in that reference, with something weaker.
 Consider a convex neighborhood $U_t$ of $\gamma(t)$ for $t\in I$ (according to Proposition~\ref{convexneigh}) and then a smaller neighborhood $V_t$ obtained using the strong causality, namely, if a causal curve starting in $V_t$ leaves $U_t$ never returns to $V_t$. Now take a locally finite collection $\{V_k\}$ of $\{V_t\}_{t\in I}$ that covers $\gamma$. Then one can basically repeat the reasoning in \cite[Lem. 3.29]{BeEhEa96}.
 \end{proof}
\begin{proposition}\label{generallimitcurve}
Let $\{\gamma_n\}$ be a sequence of (future) inextendible causal curves in a cone structure $(M,\C)$. If $p$ is an accumulation point of $\{\gamma_n\}$, then there exists a (future) inextendible causal curve $\gamma$ which is a limit curve of $\{\gamma_n\}$ with $p$ in the image of $\gamma$.
\end{proposition}
\begin{proof}
Observe that it is possible to construct a Lorentzian metric $g$ on $M$ such that all the vectors in $\C$ are timelike vectors for $g$. Indeed, consider a Lorentz-Finsler metric $G$ as in \eqref{Gmetric} and
 choose any Riemannian metric $g_R$ on the bundle $\ker (\Omega)$ such that, for all $w\in\ker(\Omega)$,  $F^2(w)> g_R(w,w)$; then define $g(v,u):=\Omega(u)\Omega(v)  -g_R\big(\pi(v),\pi(v)\big)$, where $\pi$ is the projection $\pi:TM\to \ker \Omega$. Since the future lightlike vectors of $G$ are the vectors in $\C$, it is clear that any $v\in \C$ is $g$-timelike.   Now the curves $\{\gamma_n\}$ are also causal future inextendible for $(M,g)$, and applying \cite[Prop. 3.31]{BeEhEa96}, we deduce the existence of a future inextendible limit curve $\gamma$. 
Using the convex neighborhoods obtained in Proposition~\ref{convexneigh} in the same way as in the proof of \cite[Prop. 3.31]{BeEhEa96}, 
 it follows that $\gamma$ is $\C$-causal. 
\end{proof}

\subsection{Some basic results of causality}\label{somebasic}
Propositions \ref{existsFinsler} and \ref{convexneigh} will allow us to use all the tools that Finsler geometry provides to study the causality of cone structures. As it was observed in \cite{MinConvex}, the existence of convex neighborhoods for Finsler spacetimes allows for a generalization of most of the resuls of the standard theory of causality. We will restrict  in what follows to the analysis of how to generalize the results that will be needed in the description of the causality in the presence of a Killing field.
\begin{proposition}\label{basicprops}
	Let $(M,\C)$ be a cone structure. If $p,r,q\in M$, the following properties hold:
	\begin{enumerate}[(i)]
		\item $I^+(p)$ and $I^-(p)$ are open subsets for all $p\in M$,
		\item if $p\ll r$ and $r\leq q$, then $p\ll q$,
			\item if $p\leq r$ and $r\ll q$, then $p\ll q$,
			\item  If $q\in J^+(p)\setminus I^+(p)$, then there exists a cone geodesic without conjugate points from $p$ to $q$,
			\item  If $q\in J^-(p)\setminus I^-(p)$, then there exists a cone geodesic without conjugate points from $q$ to $p$,
			\item $\bar J^+(p)=\bar I^+(p)$ and $\bar J^-(p)=\bar I^-(p)$.
		\end{enumerate}
	\end{proposition}
\begin{proof}
	Observe that we can assume that there exists a Lorent-Finsler metric $L$ as in Proposition \ref{existsFinsler}. Part $(i)$ is a straightforward consequence of  the fact that timelike curves have tangent vectors included in the open subset  $A$. Part $(ii)$ and $(iii)$ follows, for example, from  \cite[Prop. 6.5]{AJ16}. Parts $(iv)$ and $(v)$ are a consequence of \cite[Prop. 6.8]{AJ16} when $P$ is a point and taking into account that the notion of conjugate points for cone geodesics makes sense thanks to \cite[Ths. 2.5 and 3.8]{JaSo21}. Finally, $(vi)$ can be deduced by checking that $J^\pm(p)\subset \bar I^\pm(p)$ using, for example, variations of curves as in 
   \cite[Prop. 6.5]{AJ16} or, when there is a lightlike geodesic $\gamma$ from $p$ to $q\in J^\pm (p)$,  timelike geodesics  close to $\gamma$.
	\end{proof}
\begin{definition}
	A cone structure $(M,\C)$ is said to be past (resp. future) reflective if $I^+(p)\supset I^+(q)$ 
	(resp. $I^-(p)\supset I^-(q)$) implies that $I^-(p)\subset I^-(q)$ 
	(resp. $I^+(p)\subset I^+(q)$).
	\end{definition}
\begin{proposition}\label{causalcontrefl}
	A cone structure $(M,\C)$ is causally continuous if and only if it is distinguishing and future and past reflective.
\end{proposition}
\begin{proof}
It follows from the fact that outer continuity of $I^+$ (resp. $I^-$) is equivalent to future (resp. past) reflectivity. The proof goes on as in \cite[Prop 3.46]{MinSan} taking into account that in the definition of past and future reflectivity one can replace the chronological future and past by their closures (see \cite[Lemma 3.42]{MinSan}).
\end{proof}
\begin{proposition}\label{globimpclosedcones}
A globally hyperbolic cone structure $(M,\C)$ has closed causal sets $J^\pm(p)$, for all $p\in M$.
\end{proposition}
\begin{proof}
It follows the same steps as \cite[Prop. 3.71]{MinSan} using properties $(i)$ and $(vi)$ of Proposition \ref{basicprops}.
 \end{proof}
Recall that a subset $S$ is said to be acausal if no pair of points $p,q\in S$ are causally related.
\begin{proposition}\label{inextengeo}
Let $(M,\C)$ be a cone structure. Then  a  closed acausal topological hypersurface $S$ is Cauchy if and only if all inextendible cone geodesics cross it.
\end{proposition}
\begin{proof}
 It follows the same steps as \cite[Cor. 14.54]{o1983} introducing the notions of Cauchy development and future Cauchy horizons in the obvious way.
\end{proof}
\section{Wind Finslerian structures}\label{wfm}
Let us recall that the Zermelo's navigation problem poses the question of finding the time-minimizing trajectories of a ship  in the presence of a wind or a current. Namely, assume  initially that  the velocity of the ship is the same in all directions, and then  introduce a mild wind. Then  the new  velocities vectors  are given by a (non-centered) sphere which  contains the origin in the interior of its bounded region. It can be shown that the solutions to Zermelo's problem in this case can be obtained as geodesics of the Finsler metric having as indicatrix the non-centered sphere of  velocities (see \cite{BRS04}), which in turn, it is a Finsler metric of Randers type, namely, it can be expressed as
\[F(v)=\sqrt{h(v,v)}+\omega(v),\]
where $h$ is a Riemannian metric  and $\omega$ is a one-form with $h$-norm less than $1$ at every point. 
The definition of a wind Finslerian structure is motivated by the Zermelo's problem when the wind can be stronger  than the  velocity of the ship and its set of maximal possible velocities (without wind)  is   any compact,  strongly convex hypersurface containing the zero vector in the interior of its bounded region. The set of velocities in the presence of the wind is a translation of this hypersurface and it may not contain the zero vector in its bounded interior region. As a consequence this subset is not the indicatrix of  a Finsler metric.  
\begin{definition}\label{windFinsler}
Given a manifold $M$, we say that a smooth hypersurface $\Sigma$ of $TM$ is a {\em wind Finslerian structure} if
\begin{enumerate}[(i)]
\item $\Sigma_p:=\Sigma\cap T_pM$ is a compact, connected, strongly convex, embedded hypersurface in $T_pM$,
\item $\Sigma$ is transversal to the fibers of the tangent bundle, namely, if $v\in \Sigma_p$, then 
$T_v(T_pM)+T_v\Sigma=T_v(TM)$.
\end{enumerate}
\end{definition}
Observe that in this case we will use  the following notation:
\begin{enumerate}[(i)]
\item the bounded open region determined by $\Sigma_p$ in $T_pM$ is the unit ball of $\Sigma_p$ and it is denoted by $B_p$,
\item the domain $A_p$ of $\Sigma_p$ is defined as the open set of vectors $v\in T_pM$ such that there exists $\lambda>0$ such that $\lambda v\in B_p$.
\end{enumerate}
Moreover, we have to distinguish three types of points in a wind Finslerian structure (see \cite[Def. 2.11 and Prop. 2.12]{CJS14} for details).
\begin{definition}\label{windregions}
Let $(M,\Sigma)$ be a wind Finslerian structure. Then we say that $p\in M$ is a point of
\begin{enumerate}[(i)]
\item  {\em mild wind} if $0\in B_p$. In this case, $\Sigma_p$ is the indicatrix of a Minkowski norm $F$. The open subset of mild wind points is called the mild region and denoted $M_{mild}$,
\item  {\em critical wind} if $0\in \Sigma_p$. In this case, $\Sigma_p$ is the indicatrix of a Kropina type norm $F$. The closed subset of 
critical wind points is called the critical region and denoted $M_{crit}$,
\item  {\em strong wind} if $0\not\in \bar B_p$. In this case, $\Sigma_p\cap A_p$ has two connected components, one of them is the indicatrix of a conic Minkowski norm $F$, and the other one of a Lorentz-Minkowski norm $F_l$. The subset of strong wind points is called the strong region and denoted $M_l$.
\end{enumerate} 
\end{definition}
Moreover, denoting
\[A=\cup_{p\in M}A_p,\quad A_l=\cup_{p\in M_l}A_p\]
\[\text{ and } A_E=\{\text{the closure of $A_l$ in $TM_l\setminus \bf 0$} \}\cup \{ 0_p\in T_pM: p\in M_{crit}\},\]
 there is a conic Finsler metric $F:A\rightarrow(0,+\infty)$ and a Lorentz-Finsler metric $F_l:A_l\rightarrow (0,+\infty)$. Observe that
these metrics can be extended by continuity respectively to $A\cup A_E$ and  $A_E$, with the zero section as a special case (as commented in \cite[Con. 2.19]{CJS14} it is not possible to extend these metrics continuously to the zero section in the critical region, but we will choose $F(0_p)=F_l(0_p)=1$ for all $p\in M_{crit}$ by convenience).

Now we can give some definitions for curves.
\begin{definition}
Given a wind Finslerian structure $(M,\Sigma)$, with $F,F_l$ its associated pseudo-Finsler metrics, we say that a smooth curve $\gamma:[a,b]\rightarrow M$ is 
\begin{enumerate}[(i)]
\item {\em $\Sigma$-admissible} if $\dot\gamma(t)\in A\cup A_E$ for all $t\in [a,b]$,
\item {\em a wind curve} if it is $\Sigma$-admissible and 
\[F(\dot\gamma(t))\leq 1\leq F_l(\dot\gamma(t)),\quad \forall t\in [a,b],\]
\item {\em a (strictly) regular wind curve} if, in addition, the derivatives vanish at most at isolated points (do not vanish).
\end{enumerate}
These definitions can be extended to piecewise smooth curves by requiring that every smooth piece satisfies the above properties.
Moreover, the wind lengths of $\Sigma$-admissible curves are defined as
\[\ell_F(\gamma)=\int_a^b F(\dot\gamma(t)dt,\quad  \ell_{F_l}(\gamma)=\int_a^b F_l(\dot\gamma(t)dt.\]
\end{definition}
	\begin{definition}\label{sigmaballs}
		Let $x_0\in M$ and $r >  0$. The {\em forward} (resp. {\em backward}) {\em wind balls}  of center $x_0$ and radius $r$ 
		associated with the wind Finslerian structure $\Sigma$ are:
		\begin{align*}
			&B^+_{\Sigma}(x_0,r)=\{x\in M:\ \exists\   \gamma\in C^{\Sigma}_{x_0, x}, \text{ s.t. }   r=b_\gamma-a_\gamma \, \text{and} \;  \ell_F(\gamma)<r<\ell_{F_l}(\gamma)\},\\
			&B^-_{\Sigma}(x_0,r)=\{x\in M:\ \exists\   \gamma\in C^{\Sigma}_{x, x_0}, \text{ s.t. }  r=b_\gamma-a_\gamma \, \text{and} \;  \ell_F(\gamma)<r<\ell_{F_l}(\gamma)\},\\
			\intertext{where $C^{\Sigma}_{x_0, x}$ is the space of wind curves $\gamma:[a_\gamma,b_\gamma]\rightarrow M$ from $x_0$ to $x$. Moreover, we denote by     $\bar{B}^\pm_{\Sigma}(x_0,r)$, {\em the closed balls} (the closures of the above balls) and define the (forward, backward) {\em c-balls} as:} 
			&\hat{B}^+_{\Sigma}(x_0,r)=\{x\in M:\ \exists\  \gamma\in C^{\Sigma}_{x_0, x},\text{ s.t. } r=b_\gamma-a_\gamma \, \text{(so,} \;  \ell_F(\gamma)\leq r\leq\ell_{F_l}(\gamma) )  \}
			,\\
			&\hat{B}^-_{\Sigma}(x_0,r)=\{x\in M:\ \exists\ \gamma\in C^{\Sigma}_{x, x_0},\text{ s.t. } r=b_\gamma-a_\gamma \, \text{(so,} \; \ell_F(\gamma)\leq r\leq\ell_{F_l}(\gamma))  \} 
		\end{align*}
		for  $r> 0$ and, by convention for $r=0$, $\hat{B}^\pm_{\Sigma}(x_0,0)=x_0$.
	\end{definition}
	Finally, we can introduce the concept of geodesic adapted to wind Finslerian structures.
	\begin{definition}\label{extremizing}
		Let $(M,\Sigma)$ be a wind Finslerian structure.  A   wind  curve $\gamma: [a,b]\to M$,  $a<b$,   is called a
		{\em unit extremizing geodesic} if
		\begin{equation}\label{eunitgeodesic}
			\gamma(b)\in \hat{B}_\Sigma^+(\gamma(a),b-a)\setminus B_\Sigma^+(\gamma(a),b-a) .  \end{equation}
		reparametrization  of a unit extremizing geodesic. 
	\end{definition}
	\section{Cone wind triples}\label{windtriples}
	Let us observe that the notion of wind Finslerian structures in Def. \ref{windFinsler} can be extended to arbitrary vector bundles.  Indeed, it can be defined as a smooth hypersurface $\Sigma$ satisfying that its intersection with every fiber is a compact, connected, strongly compact, embedded hypersurface in the fiber, generalizing $(i)$ in Def. \ref{windFinsler}, and an adaptation of $(ii)$ by requiring that $\Sigma$ is transversal to the tangent space to the fiber. 
	\begin{proposition}\label{conewind}
	Given a cone structure $(M,\C)$, a non-vanishing one-form $\Omega$ with $\ker(\Omega)$ spacelike and an arbitrary vector field $T$ transversal to $\ker(\Omega)$, we can define a wind Finslerian structure in the vector subbundle $\ker(\Omega)$, by defining $\Sigma$ as the vectors $w\in \ker(\Omega)$ such that  $v=T+w\in \C$.
	\end{proposition}
	\begin{proof}
	The fact that $\Sigma\cap \ker(\Omega_p)$ is a compact, strongly convex, connected, embedded hypersurface of $\ker(\Omega_p)$ is a consequence of the properties of cone structures (see \cite[Prop. 2.6 (iii)]{JS18} and take into account that $\ker(\Omega_p)$ is transversal to $\C_p$ because it is spacelike). Finally, to prove that $\Sigma$ is transversal to $\ker(\Omega_p)$, observe that by definition of cone structure $\C$ is transversal to $T_pM$ and then if $w\in \Sigma$, we have by definition that $v=w+T_p\in \C$ and  $T_vTM=T_v\C+T_v(T_pM)$. If $\Sigma$ were not transversal to $\ker(\Omega_p)$, we would have that
\[T_w\Sigma+T_w(\ker(\Omega_p))\subsetneq T_w (\ker(\Omega)),\]	
and taking into account that $T_v\C={\rm span}\{v\}+T_w\Sigma$, it follows that
\begin{multline*}
	T_v\C+T_v(T_pM)={\rm span}\{v\}+T_w\Sigma+T_w(\ker(\Omega_p))\\ \subsetneq {\rm span}\{v\}+T_w(\ker(\Omega))=T_v(TM),\end{multline*}
	and then $\C$ wouldn't be transversal to $T_pM$, a contradiction.
	\end{proof}
\begin{definition}\label{conewindtriple}
We call a triple $(\Omega, T, \Sigma)$ associated with a cone structure $(M,\C)$, with $\Omega$, $T$ and $\Sigma$ as in the statement of Proposition~\ref{conewind}, a {\em cone wind triple} associated with $(M,\C)$.
\end{definition}
\section{Cone Killing fields of cone structures and \cstk splittings}\label{cstk}
We introduce the concept of a cone Killing vector field which naturally extends to cone structures the notion of a conformal vector field in a spacetime. 
\begin{definition}\label{killing}
We say that a vector field $K$ of $M$ is {\em cone Killing} for a certain cone structure $\mathcal{C}$ on $M$ if its flow locally preserves $\C$, namely whenever the flow is defined for a certain instant $t$, $\phi_t:U\rightarrow M$, the differential of $\phi_t$ maps $\mathcal{C}_p$ into $\mathcal{C}_{\phi_t(p)}$.
\end{definition}
If a cone structure $(M,\C)$ admits a complete cone Killing vector field $K$ and a spacelike hypersurface $S$ (in the sense that its tangent bundle is made by spacelike subspaces according to the definitions given at the beginning of \S~\ref{causalityforcones}) transversal to $K$, then a  cone structure is naturally defined on $\R\times S$ as follows:
\begin{proposition}\label{windFinslercone}
Let $(M,\C)$ be a cone structure, $K$ a complete cone Killing field for $\C$ and $S$  a spacelike hypersurface which is transversal to $K$ and intersects only once every orbit of $K$ (a spacelike section). Then $M$ is diffeomorphic to $\R\times S$ and this diffeomorphism induces a cone structure $\C^S$ on $\R\times S$. Moreover,  there is a wind Finslerian structure  $(S,\Sigma)$ that when considered in every $\{t_0\}\times S$ generates the cone wind triple obtained in Proposition~\ref{conewind} associated with $(\R\times S,\C^S)$ with $T=K$, $\ker(\Omega)=TS$ and $\Omega(K)=1$.
\end{proposition}
\begin{proof}
The first claim about the diffeomorphism between $M$ and $\R\times S$ is obtained by using the flow of $K$. To obtain the cone wind triple, consider the one-form $\Omega$ determined by having  the tangent space to $S$ as kernel and satisfying that $\Omega(K)=1$. Choosing in addition $T=K$,  it follows that the cone wind triple obtained in Proposition~\ref{conewind} has as third element  in the triple a wind Finslerian structure  on $S$ which is preserved  by the flow of $K$.
\end{proof}

The product manifold with the cone structure $(\R\times S,\C^S)$ obtained in the last proposition will be referred to in the following as a {\em cone structure with a space-transverse Killing field splitting} (a \cstk splitting in short). We will see that the causal properties of this cone structure can be described in terms of the cone wind triple. Summing up, we will say that a cone structure  $(\R\times S,\C)$ is a \cstk splitting if $\partial_t$ (being $t$ the first coordinate of $\R\times S$) is a cone Killing field and the slices $t={\rm const}$ are spacelike. 
Let us see that $t$ is a time function of $(\R\times S,\C)$. 
\begin{proposition}\label{t}
The function $t:\R\times S\to \R$ in a \cstk splitting $(\R\times S,\C)$ is a (smooth) time function.
\end{proposition}	
\begin{proof}
Let us first show that $t$ is strictly  increasing along any smooth future causal curve $\gamma:I\to \R\times S$. If not, there must exist at least one instant $s\in I$ such that $\de t(\dot\gamma(s))\leq 0$. Let $w\in TS$ and $a\in \R$ such that $\dot \gamma(s)=a\partial_t|_{\gamma(s)}+w$ so that  $\de t(\dot\gamma(s))=a$. According to Proposition~\ref{conewind}, we have that $a>0$, a contradiction. Let now $\gamma$ be a continuous future causal curve. In every small enough interval $I_\varepsilon$, the image of $\gamma|_{I_\varepsilon}$ is contained in a convex neighborhood and $s,t\in I_\varepsilon$, with $s<t$, implies that there exists a smooth causal curve from $\gamma(s)$ to $\gamma(t)$. Applying the first part we deduce that $t(\gamma(s))<t(\gamma(t))$, which implies that $t\circ \gamma$ is locally strictly increasing, and then strictly increasing.
\end{proof}	
Finally, we will give a very direct proof of the strongly causal property of a \cstk splitting.
\begin{proposition}\label{strongcaus}
A \cstk splitting $(\R\times S,\C)$ is  strongly causal.
\end{proposition}
\begin{proof}
Given a point $p\in \R\times S$ and a neighborhood $U$, consider a compact neighborhood
$\tilde U$ endowed with a flat time-oriented Lorentzian metric $g$ having $\partial_t$ as a static Killing vector field and  with bigger future lightcones than $\C$.   Then choosing $V$ small enough such that the causal curves leaving $\tilde U$ reach a $t$-level bigger than $t^{-1}(V)$, we obtain the strongly causal property.
\end{proof}
\section{Finsler-Kropina type metrics and cone structures with a causal cone Killing field}\label{FinslerKropina}
Let us consider now the case in that the cone structure $(M,\C)$  admits a future causal  cone Killing field and a spacelike section $S$. By Proposition~\ref{windFinslercone} this means that there is a wind Finslerian structure $\Sigma$ in $S$ associated with $(M,\C)$. Since in this case, the vector field associated with the  cone wind triple, $T=K$, is causal, it follows that the zero vector is in the bounded region delimited by $\Sigma_p$, for each $p\in M$. As a consequence,  the region $M_l$ (recall part $(iii)$ of Definition~\ref{windregions}) is empty and the metric $F_l$ is not defined. Therefore, there is only one conic Finsler metric with a strongly convex indicatrix that in some points can contain the origin. 
We will refer to this case as a {\em Kropina type metric}.
As the metric $F$ is a classical Finsler metric in the region of mild wind and of Kropina type in the critical region, we will say that it is a {\em Finsler-Kropina type metric}.  Observe that, in this case, $A_p$ is a half-space when $F$ is of Kropina type or $A_p=T_pS$ when $F$ is a Finsler metric.

We say that a piecewise curve is {\em $F$-admissible} if its velocity belongs to $A$ everywhere and we denote by $\Omega^A_{p, q}$ the space of $F$-admissible curves from $p$ to $q$.  Moreover, it is possible to define an  {\em $F$-separation} as
\begin{equation*}
d_F(p,q)=
\begin{cases}
\inf_{\gamma\in \Omega^A_{p, q}}\ell_F(\gamma)&\text{if $\Omega^A_{p, q}\not=\emptyset$},\\
+\infty &\text{if $\Omega^A_{p, q}=\emptyset$},
\end{cases}
\end{equation*}
and then the forward $B^+_F(p,r)=\{q\in S: d_F(p,q)<r\}$ and backward $B^-_F(p,r)=\{q\in S: d_F(q,p)<r\}$ balls.   We will also 
say that $p$ {\em precedes} $q$, denoted $p\prec q$, if $\Omega^A_{p, q}\not=\emptyset$.
 In the next, we will generalize the results in \cite[\S4]{CJS14}.
\begin{proposition}\label{bolas} 
	 Let $(\R\times S,\C)$ be a \cstk splitting with future causal $K=\partial_t$ and $F$ be  the Finsler-Kropina type metric in $S$ associated with it. Then
		$$(t_0,x_0)\ll (t_1,x_1) \quad  \Leftrightarrow \quad d_F(x_0,x_1)<t_1-t_0 ,$$
		for every $x_0,x_1\in S$ and $t_0,t_1\in\R$. Therefore:
		\begin{equation}\begin{array}{l}
				\label{bolas0} I^+(t_0,x_0)= \{(t,y): d_F(x_0,y)<t-t_0\}, \\
				I^-(t_0,x_0)= \{(t,y): d_F(y,x_0)<t_0-t\}.
		\end{array}\end{equation}
		Equivalently, considering $d_F$-forward and backward balls
		\begin{equation*}I^+(t_0,x_0)= \cup_{s> 0}\{t_0+s\}\times B_F^+(x_0,s), \quad
			I^-(t_0,x_0)= \cup_{s> 0}\{t_0-s\}\times B_F^-(x_0,s).
		\end{equation*}
	\end{proposition}
	\begin{proof}
	It is straightforward to prove that by the same definition of causality and the metric $F$ that  $(\tau,v)$ is future timelike if and only if $\tau>F(v)$, namely, because  $(1,v)$ is future timelike if and only if $v$ is contained in the unit ball of $F$. Then the proof follows the exactly same steps as in \cite[Prop. 4.1]{CJS14}.
	\end{proof}
	The next result is proved in \cite[Prop. 4.2]{CJS14}.
	\begin{proposition} \label{pls} The $F$-separation  associated with any conic
		pseudo-Finsler metric $F$ is upper semi-continuous, i.e., if $
		x_n\rightarrow x$, $ y_n\rightarrow y$, then $$\limsup_n
		d_F(x_n,y_n) \leq d_F(x,y).$$ In particular, if $x\prec y$, then $x_n
		\prec y_n$ for large $n$.
	\end{proposition}
	\begin{proposition}\label{ctf}
	Let $(\R\times S,\C)$ be a \cstk splitting with future  causal $K=\partial_t$  and $F$, the Finsler-Kropina type metric in $S$ associated with it, then  the function  
		$$
		\tau_F:
		S\times S \rightarrow [0,+\infty], \quad \quad \tau_F(x,y):=
		\inf\{t\in\R : \, (0,x)\ll (t,y)\}$$ is equal to the $F$-separation
		function $d_F$.
	\end{proposition}
	\begin{proof}
It follows from Proposition \ref{bolas}.
	\end{proof}
	\begin{theorem}\label{tcontdf}
		The $F$-separation $d_F: S\times S \rightarrow [0,+\infty]$
		associated with any Finsler-Kropina type metric is continuous away from the
		diagonal $D=\{(x,x): x\in S\}\subset S\times S$.
	\end{theorem}
	\begin{proof}
We will construct a cone structure on $\R\times S$ and use  Propositions~\ref{bolas} and \ref{ctf}. As a first step,  consider a one-form $\Omega$ with kernel the tangent space to the slices $\{t_0\}\times S$ for all $t_0\in\R$, and such that $\Omega(\partial_t)=1$. The cone wind triple $(\Omega, \partial_t, \Sigma)$, with $\Sigma$ the indicatrix of $F$, determines a cone structure $(\R\times S,\C)$ which has $\partial_t$ as a future causal cone Killing vector field (recall Sections~\ref{windtriples} and \ref{cstk}). 
By applying Proposition~\ref{generallimitcurve} and part $(iv)$ of Proposition~\ref{basicprops},
the proof proceeds using the same steps as those in the proof of \cite[Th. 4.5]{CJS14}. 
	\end{proof}
\begin{proposition} 
		The $F$-separation $d_F$ associated with  a Finsler-Kropina type metric is discontinuous at $(x_0,x_0)$ if $d_F(x_0,x_0)>0$.  Moreover, 
		\begin{itemize}
			\item[(i)] the property $d_F(x_0,x_0)>0$ occurs if there exists a neighborhood $U$ of  
			$x_0$ such that no admissible loop contained in $U$ exists, i.e.  $y\not\prec_U y$, for all $y\in U$;
			in particular for any Kropina type metric  as in \eqref{kropina}  such that  $\ker \beta$ is  integrable,  i.e. $\beta\wedge d\beta=0$; 
			\item[(ii)]for any Kropina type norm on a vector space, $d_F(x,x)=\infty$ for all $x\in V$.
		\end{itemize}
	\end{proposition}
\begin{proof}
As in \cite[Prop. 4.6]{CJS14}, because it depends essentially on the points $x_0$ such that  $F$ restricted to $T_{x_0} S$ is a Kropina type norm.
	\end{proof}
	The discontinuity of $d_F$ at the diagonal necessitates the addiction of the center of the ball at hand:
	\begin{corollary}\label{cclballs}
		The closed forward (resp. backward) $d_F$-balls, defined as the closures of the corresponding open balls, satisfy,  for $r>0$: 
		\[\bar B_F^+(x,r)= \{y\in S:
		d_F(x,y)\leq r\}\cup \{x\}\]
		\[\hbox{(resp.} \; \bar B_F^-(x,r)= \{y\in S: d_F(y,x)\leq r\}\cup
		\{x\}).
		\]
	\end{corollary}
	\begin{proof}
		As in \cite[Cor. 4.8]{CJS14}.
		\end{proof}
	\subsection{Ladder of causality}
	\begin{theorem}\label{kropinaLadder} Consider a \cstk  splitting $(\R\times S,\C)$  with $K=\partial_t$ causal and associated Finsler-Kropina type metric $F$ on $S$. Then, $(\R\times S,\C)$  is causally
		continuous, and
		\begin{enumerate}[(i)]  \item the following assertions are
			equivalent:
			
			(i1) $(\R\times S,\C)$ is causally simple 
			
			(i2) $(S,F)$ is {\em convex}, in the sense that  for
			every $x,y\in S$, $x\neq y$,  with $d_F(x,y)<+\infty$, there exists a geodesic
			$\gamma$ from $x$ to $y$ such that $\ell_F(\gamma )=d_F(x,y)$.
			
			(i3) $J^+(p)$ is closed for all $p\in \R\times S$.
			
			(i4) $J^-(p)$ is closed for all $p\in \R\times S$.
			
			\item $(\R\times S,\C)$ is globally hyperbolic (i.e. it is causal and all the intersections $J^+(p) \cap  J^-(q)$ are compact) if and only if  $\bar B_F^+(x,r_1)\cap \bar B_F^-(y,r_2)$ is compact for every $x,y\in
			S$ and
			$r_1,r_2>0$.
			\item The following assertions are
			equivalent:
			
			(iii1) A slice $S_t=\{(t,x): x\in \R\times S\}$ (and, then all the slices) is a
			spacelike Cauchy hypersurface i.e., it is crossed exactly once
			by any
			inextendible timelike curve (and, then, also by any causal one),
			
			(iii2) the closures $\bar{B}_F^+(x,r)$, $\bar{B}_F^-(x,r)$ are compact
			for all $r>0$ and  $x\in S$,
			
			(iii3)   $F$ is forward and backward geodesically complete.
		\end{enumerate}
	\end{theorem}
\begin{proof}
	If follows the same steps as \cite[Th. 4.9]{CJS14} taking into account Propositions \ref{causalcontrefl},   \ref{inextengeo}, \ref{t} and \ref{strongcaus}.
	\end{proof}
	As a straightforward consequence of Theorem~\ref{kropinaLadder}  and the  implications from causality theory
	{\em (iii1)} $\Rightarrow$ global hyperbolicity $\Rightarrow$ {\em (i1)} (applying Proposition \ref{globimpclosedcones}), one has the following version of Hopf-Rinow theorem.
	\begin{corollary}\label{cRandersKropinaHopfRinow}
		For any Finsler-Kropina type metric $F$ on a manifold $S$, the
		forward (resp. backward) geodesic completeness of $d_F$ is equivalent to the
		compactness of the forward closed balls $\bar B^+_F(x,r)$ (resp. backward closed balls $ \bar B^-_F(x,r)$)  for every $x\in S, r>0$.  Moreover,  any of these properties implies the compactness of the intersection between any pair 
		of forward and backward closed balls. Finally,  the  last property  implies the convexity of $(S,F)$, in the sense of Theorem~\ref{kropinaLadder}.
	\end{corollary}
Observe that Finlser pp-waves provide examples of Finsler spacetimes having a lightcone admitting a cone Killing field of lightlike type (see \cite[\S6]{AJW23} for explicit examples).
	\section{Cone structures with an arbitrary cone  Killing field}\label{arbitraryKilling}
	In this section we consider the general case of a cone Killing field $K$  with no restriction on its pointwise  causal character. 
	\begin{proposition}\label{bolas2} Let $(\R\times S,\C)$ be a \cstk  splitting and $\Sigma$, the wind Finslerian structure in $S$ associate with it. Then:
		\begin{align*}
			&I^+(t_0,x_0)= \cup_{s> 0}\{t_0+s\}\times B_{\Sigma}^+(x_0,s),\\
			&I^-(t_0,x_0)= \cup_{s> 0}\{t_0-s\}\times B_{\Sigma}^-(x_0,s),\\
			&J^+(t_0,x_0)= \cup_{s\geq  0}\{t_0+s\}\times \hat{B}_{\Sigma}^+(x_0,s),\\
			&J^-(t_0,x_0)= \cup_{s\geq  0}\{t_0-s\}\times \hat{B}_{\Sigma}^-(x_0,s).
		\end{align*}
	\end{proposition}
\begin{proof}
	Follow the same steps as in \cite[Prop. 5.1]{CJS14} using \cite[Prop. 6.5]{AJ16} rather than \cite[ Prop. 10.46]{o1983}.
	\end{proof}
	\begin{corollary}\label{horismos}
		Two distinct points $(t_0,x_0), (t_1,x_1)\in \R\times S$ are horismotically related
		if and only if
		$x_1 \in  \hat B^+_\Sigma(x_0,t_1-t_0)\setminus  B^+_\Sigma(x_0,t_1-t_0)$.
		
		In this case,  there exists a cone geodesic $\rho:[t_0,t_1]\rightarrow \R\times S$, $\rho(s)=(s,x(s))$ from $(t_0,x_0)$ to $(t_1,x_1)$ and 
		such that $x$ is a  unit extremizing 
		geodesic   of $\Sigma$  from $x_0$ to $x_1$ with $\ell_F(x)=t_1-t_0$ or $\ell_{F_l}(x)=t_1-t_0$ (or both).
	\end{corollary}
	\begin{proof}
		As in \cite[Prop. 5.1]{CJS14} using again \cite[Prop. 6.5]{AJ16} rather than \cite[Prop. 10.46]{o1983}.
	\end{proof}
	\begin{lemma}\label{stronly}
		Given a \cstk splitting $(\R\times S,\C)$ and $z_0=(t_0,x_0)\in\R\times S$ there exists a convex neighborhood  $U$ of $z_0$,  a neighborhood $V$ of $z_0$ contained in $U$  and some small $\varepsilon>0$  such that
		$J^+(z)\cap \{(t,x)\in \R\times S:  t\in [t(z), t_0+\varepsilon) \}\subset U$ for every $z\in V$. 
	\end{lemma}
\begin{proof}
	As the one in \cite[Lemma 5.4]{CJS14}.
	\end{proof}
		\begin{lemma}\label{limitcurve} Let $( M,\C)$ be a cone structure endowed
		with a time function $t: M\rightarrow \R$.
		\begin{enumerate}[(i)]
			\item   Consider a sequence  of inextendible causal curves $\{\gamma_n\}$ parametrized by the time $t$ and assume that there exists a convergent sequence $\{t_n\}$ such that $\gamma_n(t_n)$ converges to $z_0$. 
			Then there exists an (inextendible, causal) limit curve $\gamma$ 
			through $z_0$ parametrized by the time $t$,  and   a  subsequence  $\gamma_{n_k}$  
			such  that,  whenever    the intersection
			of $\gamma$ with the slice $S_{t_0}:=\{z\in  M: t(z)=t_0\}$ is  not empty for   $t_0\in \R$,
			then all the curves $\gamma_{n_k}$ but a finite number intersect $S_{t_0}$ and  $\gamma(t_0)=\lim_k\gamma_{n_k}(t_0)$.
			\item  Let  $\gamma_n$ be a sequence of causal
			curves   and, for each $n\in\N$,  $z_n\leq w_n$ be two points on  $\gamma_n$.  If $z_n\to z$, $w_n\to w$,  $z\neq w$, and the    intersection of the slice
			$S_{t_0}$ with the images of all $\gamma_n$  lies in a
			compact subset for any $t_0 \in   (t(z),t(w)) $,  then any (inextendible) limit
			curve $\gamma$ of the sequence starting at   $z$   arrives at   $w$. 
		\end{enumerate}
	\end{lemma}
	\begin{proof}
		As in \cite[Lemma 5.7]{CJS14} taking into account \S\ref{limitcurves}.
		\end{proof}
	\begin{proposition}\label{lreflect} For any $p=(t_0,x_0), q=(t_1, x_1)$ in a \cstk splitting:
		\begin{enumerate}[(i)]
			\item $I^+(p)\supset I^+(q)$ if and only if $x_1\in
			\bar{B}^+_{\Sigma}(x_0,t_1-t_0)$, and
			\item  $I^-(p)\subset I^-(q)$ if and only if $x_0\in
			\bar{B}^-_{\Sigma}(x_1,t_1-t_0)$.
		\end{enumerate}
		\smallskip
		\noindent Moreover,
		\begin{align*}
			&\bar J^+(t_0,x_0)= \left(\cup_{s> 0}\{t_0+s\}\times \bar{B}_{\Sigma}^+(x_0,s)\right)\,\cup\,\{(t_0,x_0)\},\\
			&\bar J^-(t_0,x_0)= \left(\cup_{s> 0}\{t_0-s\}\times \bar{B}_{\Sigma}^-(x_0,s)\right)\,\cup\,\{(t_0,x_0)\}.
		\end{align*}
	\end{proposition}
	\begin{proof}
			If follows the same steps as \cite[Th. 5.8]{CJS14} taking into account part $(vi)$ of Proposition \ref{basicprops}.
			\end{proof}
	\begin{theorem}\label{generalK} Consider a \cstk  splitting $(\R\times S,\C)$ with associated wind Finslerian structure $\Sigma$ on $S$. Then, $(\R\times S,\C)$  is stably causal and
		\begin{enumerate}[(i)]
			\item  $(\R\times S,\C)$ is causally continuous if and only if
			$\Sigma$ satisfies the following property: given any pair  of points
			$x_0,x_1$ in $S$ and $r>0$, $x_1\in \bar{B}^+_{\Sigma}(x_0,r)$ if
			and only if $x_0\in \bar{B}^-_{\Sigma}(x_1,r)$.
			
			\item  $(\R\times S,\C)$ is causally simple
			if and only if $(S,\Sigma)$ is  w-convex, namely, the c-balls $\hat{B}^+_{\Sigma}(x,r)$ 
			and  $x_0\in \hat{B}^-_{\Sigma}(x,r)$ are closed for all $x\in S$ and $r>0$.
			\item  The following assertions are equivalent:
			\begin{enumerate}
				\item[(iii1)] $(\R\times S,\C)$ is globally hyperbolic.
				\item[(iii2)]$\hat{ B}^+_{\Sigma}(x,r_1)\cap
				\hat{B}^-_{\Sigma}(y,r_2)$ is compact for every $x,y\in
				S$ and $r_1,r_2>0$.
				\item[(iii3)]  $\bar{ B}^+_{\Sigma}(x,r_1)\cap \bar{B}^-_{\Sigma}(y,r_2)$
				is compact for every $x,y\in S$ and
				$r_1,r_2>0$.
			\end{enumerate}
			\item The following assertions are
			equivalent:
			\begin{enumerate}
				\item[(iv1)] A slice $S_t$ (and, then every slice) is a
				spacelike Cauchy hypersurface.
				\item[(iv2)] All the c-balls $\hat B_{\Sigma}^+(x,r)$ and
				$\hat B_{\Sigma}^-(x,r)$, $r>0$, $x\in S$, are compact.
				\item[(iv3)] All the (open) balls $B_{\Sigma}^+(x,r)$ and
				$B_{\Sigma}^-(x,r)$, $r>0$, $x\in S$, are precompact.
				\item[(iv4)]   $\Sigma$ is forward and backward geodesically complete.
			\end{enumerate}
		\end{enumerate}
	\end{theorem}
\begin{proof}
	If follows the same steps as \cite[Th. 5.9]{CJS14} taking into account Propositions \ref{causalcontrefl},  \ref{inextengeo}, \ref{t} and \ref{strongcaus}.
\end{proof}
Observe that in \cite[Th. 3.23]{JS_proc}, some of the characterization in the above theorem are rewritten in terms of  (an extension) of the separation $d_F$ introduced in \S \ref{FinslerKropina}.
\section*{Acknowledgments}
We thank M. S\'anchez for several fruitful discussions and for having suggested us the name ``cone Killing vector field''  for a vector fields preserving a cone structure according to Definition~\ref{killing}.

\smallskip

\noindent EC is partially supported by European Union - Next Generation EU - PRIN 2022 PNRR {\it ``P2022YFAJH Linear and Nonlinear PDE's: New directions and Application''},  by  MUR under the Programme ``Department of Excellence'' Legge 232/2016 
(Grant No. CUP - D93C23000100001) and by GNAMPA INdAM - Italian National Institute of High Mathematics.

\smallskip

\noindent MAJ was partially supported by the project PID2021-124157NB-I00, funded by MCIN/AEI/10.13039/501100011033/``ERDF A way of making Europe", and also by Ayudas a proyectos para el desarrollo de investigaci\'{o}n cient\'{i}fica y t\'{e}cnica por grupos competitivos (Comunidad Aut\'{o}noma de la Regi\'{o}n de Murcia), included in the Programa Regional de Fomento de la Investigaci\'{o}n Cient\'{i}fica y T\'{e}cnica (Plan de Actuaci\'{o}n 2022) of the Fundaci\'{o}n S\'{e}neca-Agencia de Ciencia y Tecnolog\'{i}a de la Regi\'{o}n de Murcia, REF. 21899/PI/22.

\end{document}